\documentclass[a4paper,10pt,psamsfonts]{amsart}

\usepackage[utf8]{inputenc} 

\usepackage{amsmath}
\usepackage{amstext}
\usepackage{amsfonts}
\usepackage{amsthm}
\usepackage{amssymb}

\usepackage{hyperref}
\hypersetup{colorlinks}

\usepackage{color}

\newcommand{\R}{{\mathbb{R}}}

\newcommand{\E}{{\mathbb{E}}}
\newcommand{\N}{{\mathbb{N}}}

\newcommand{\F}{{\mathcal{F}}} 
\renewcommand{\P}{{\mathbb{P}}} 

\newcommand{\diff}[1]{\,\mathrm{d}#1}

\newcommand{\ee}{\mathrm{e}}

\newcommand{\triple}{{\vert\kern-0.25ex\vert\kern-0.25ex\vert}}

\theoremstyle{plain}

\newtheorem{definition}{Definition}
\newtheorem{theorem}[definition]{Theorem}
\newtheorem{lemma}[definition]{Lemma}

\newtheorem{prop}[definition]{Proposition}

\theoremstyle{definition}
\newtheorem{remark}[definition]{Remark}

\begin{document}

\title[A discrete stochastic Gronwall Lemma]
{A discrete stochastic Gronwall Lemma}

\author[R.~Kruse]{Raphael Kruse}
\address{Raphael Kruse\\
Technische Universit\"at Berlin\\
Institut f\"ur Mathematik, Secr. MA 5-3\\
Stra\ss e des 17.~Juni 136\\
DE-10623 Berlin\\
Germany}
\email{kruse@math.tu-berlin.de}

\author[M.~Scheutzow]{Michael Scheutzow}
\address{Michael Scheutzow\\
Technische Universit\"at Berlin\\
Institut f\"ur Mathematik, Secr. MA 7-5\\
Stra\ss e des 17.~Juni 136\\
DE-10623 Berlin\\
Germany}
\email{ms@math.tu-berlin.de}


\begin{abstract}
  We derive a discrete version of the stochastic Gronwall Lemma
  found in [Scheutzow, IDAQP, 2013]. The proof is based on a corresponding
  deterministic version of the discrete Gronwall Lemma and an inequality
  bounding the supremum in terms of the infimum for time discrete martingales.
  As an application the proof of an a priori estimate for the backward
  Euler-Maruyama method is included.
\end{abstract}

\maketitle

\section{Introduction}
\label{sec:intro}

The Gronwall Lemma is an often used tool in classical analysis for deriving a
priori and stability estimates of solutions to differential equations. It is
named after T.~H.~Gr\"onwall and originated in its differential form from his
work \cite{gronwall1919}. Besides the integral version in \cite{bellman1943}
many more variations of the Gronwall Lemma have been introduced with a wide
area of applications, for example, in ordinary differential equations, partial
differential equations, integral equations, and stochastic analysis.
Similarly, discrete versions of the Gronwall Lemma are often applied in order
to estimate the growth of solutions to time discrete difference equations, such
as numerical approximations of differential equations. For instance, we refer
to \cite{emmrich1999} and the references therein. The purpose of this 
paper is the derivation of the following time discrete version of the stochastic
Gronwall Lemma from \cite{scheutzow2013}:

\begin{theorem}
  \label{th:gronwall}
  Let $(M_n)_{n \in \N_0}$ be an $(\F_n)_{n \in
  \N_0}$-martingale satisfying $M_0 = 0$ on a filtered probability space
  $(\Omega, \F, (\F_n)_{n \in \N_0}, \P)$. Let $(X_n)_{n \in \N_0}$, $(F_n)_{n
  \in \N_0}$, and $(G_n)_{n \in \N_0}$ be sequences of nonnegative and adapted
  random variables with $\E[X_0] < \infty$ such that 
  \begin{align}
    \label{eq:cond}
    X_n \le F_n + M_n + \sum_{k = 0}^{n-1} G_k X_k, \quad \text{ for all } n
    \in \N_0.
  \end{align}
  Then, for any $p \in (0,1)$ and $\mu, \nu \in [1,\infty]$ with
  $\frac{1}{\mu} + \frac{1}{\nu} = 1$ and $p \nu < 1$, it holds true that
  \begin{align}
    \label{eq:gron1}
    \E \big[ \sup_{0 \le k \le n} X_k^p \big] \le 
    \Big( 1 + \frac{1}{1 - \nu p} \Big)^{\frac{1}{\nu}}
    \Big\| \prod_{k = 0}^{n-1} (1 + G_k)^p \Big\|_{L^{\mu}(\Omega)}
    \big( \E \big[ \sup_{0 \le k \le n} F_k \big] \big)^{p}
  \end{align}
  for all $n \in \N_0$. In particular, if $(G_n)_{n \in \N_0}$ is a
  deterministic sequence of nonnegative real numbers, then for any $p \in
  (0,1)$ it holds true that 
  \begin{align}
    \label{eq:gron2}
    \E \big[ \sup_{0 \le k \le n} X_k^p \big] \le 
    \Big( 1 + \frac{1}{1 - p} \Big)
    \Big( \prod_{k = 0}^{n-1} (1 + G_k)^p \Big)
    \big( \E \big[ \sup_{0 \le k \le n} F_k \big] \big)^{p}
  \end{align}
  for all $n \in \N_0$.
\end{theorem}

The main novelty of Theorem~\ref{th:gronwall} and its continuous time
counter-part in \cite{scheutzow2013} is the presence of a martingale
term on the right hand side of Equation~\eqref{eq:cond}. In this situation
deterministic versions of the Gronwall Lemma usually require to first take 
expectation in Equation~\eqref{eq:cond} in order to discard the
centered martingale from the inequality. However, this line of arguments then
often results in weaker estimates in the sense that taking the supremum with
respect to $k$ would occur outside the expectation on the left hand side of
Equations~\eqref{eq:gron1} and \eqref{eq:gron2}. 

We emphasize that the estimates in Equations~\eqref{eq:gron1} and
\eqref{eq:gron2} are uniform with respect to the martingale $(M_n)_{n \in
\N_0}$. The price we have to pay for this uniformity is the restriction of the
parameter $p$ to the interval $(0,1)$. As already indicated in Remark~3 in
\cite{scheutzow2013} the martingale inequality in Lemma~\ref{lem:martineq}
cannot be extended to $p \ge 1$. Instead one could try to apply, for instance, 
Burkholder-Davis-Gundy-type inequalities resulting in 
the appearance of the quadratic variation of the martingale
on the right hand side of the estimates.

In addition, it is worth to take note of the following subtle difference between
Theorem~\ref{th:gronwall} and its continuous time counter-part in
\cite{scheutzow2013}: On the right hand side of Equations~\eqref{eq:gron1} and
\eqref{eq:gron2} we have the $p$-th power of the expectation of $\sup_{0 \le k
\le n} F_k$. In \cite[Theorem~4]{scheutzow2013} the order of the $p$-th power
and the expectation is reversed resulting in a sharper estimate. The reason for
this difference lies in the martingale inequality in Lemma~\ref{lem:martineq}
which for discrete time martingales only holds true in the weaker form used in
this paper. Compare further with \cite[Remark~3]{scheutzow2013}.

The proof of Theorem~\ref{th:gronwall} is mostly based on two ingredients: 
The first is a discrete version of the classical Gronwall Lemma which is found
in Lemma~\ref{lem:detGronw} below. The second ingredient is an
inequality stated in Lemma~\ref{lem:martineq} that relates the $L^p$-norm, $p
\in (0,1)$, of the supremum of a time discrete martingale to its infimum.
Lemma~\ref{lem:martineq} therefore is the discrete time counter-part of
\cite[Proposition~1]{scheutzow2013}. A further version of
the latter with optimal constant is also found in \cite{banuelos2013}. For all
details of the proof we refer to Section~\ref{sec:proof}. 

As already mentioned, discrete versions of the Gronwall Lemma are often used in
order to derive \emph{a priori} estimates for numerical approximations of
differential equations. To this end we demonstrate in Section~\ref{sec:appl}
how Theorem~\ref{th:gronwall} can be applied in order to estimate the
$L^p$-norm, $p \in (0,2)$, of the backward Euler-Maruyama method for stochastic
differential equations under rather mild conditions on the coefficient
functions, namely continuity and a global coercivity condition (see
Equation~\eqref{eq:coerc} below).

\textbf{Notation:} Throughout this paper we use the convention that sums over
empty index sets are equal to zero and products over empty index sets are equal
to one. Further, we let $\N := \{1, 2, \ldots \}$ denote the set of all
positive integers and define $\N_0 := \N \cup \{0\}$. As usual, we write $a
\wedge b := \min(a,b)$ and $a \vee b := \max(a,b)$ for all $a, b \in \R$.
Finally, for an arbitrary sequence $(F_n)_{n \in \N_0}$ of random variables we
set  
\begin{align*}
  F_n^\ast := \sup_{0 \le k \le n} F_k.
\end{align*}

\section{Proof of the discrete stochastic Gronwall Lemma}
\label{sec:proof}

In this section we present a proof of Theorem~\ref{th:gronwall}. As already
indicated in the introduction, we first state a corresponding deterministic
version of the discrete Gronwall Lemma. For completeness we include a
proof based on a presentation by John
M.~Holte\footnote{\url{http://homepages.gac.edu/~holte/publications/gronwallTALK.pdf}}.
Then, we derive a discrete time version of a martingale inequality from
\cite{scheutzow2013} that gives a bound for the supremum of the martingale in
terms of its infimum. 

\begin{lemma}
  \label{lem:detGronw}
  Consider real-valued sequences $(f_n)_{n \in \N_0}$, $(g_n)_{n \in \N_0}$,
  and $(y_n)_{n \in \N_0}$. Assume that $(g_n)_{n \in
  \N_0}$ is nonnegative. If we have
  \begin{align}
    \label{eq:ineq1}
    y_n \le f_n + \sum_{k = 0}^{n-1} g_k y_k, \quad \text{for all } n \in \N_0,
  \end{align}
  then it also holds true that
  \begin{align}
    \label{eq:ineq2}
    y_n \le f_n + \sum_{k = 0}^{n-1} f_k g_k \prod_{j = k
    +1}^{n-1} (1 + g_j), \quad \text{for all } n \in \N_0.
  \end{align}  
\end{lemma}

\begin{proof}
  Obviously, the assertion is true for $n = 0$. Now let $n > 0$ and assume that
  \eqref{eq:ineq2} is satisfied for all $0 \le k < n$. Then, by inserting
  \eqref{eq:ineq2} into \eqref{eq:ineq1} for all $k < n$ we obtain
  \begin{align*}
    y_n &\le f_n + \sum_{k = 0}^{n-1} g_k y_k\\
    &\le f_n + \sum_{k = 0}^{n-1} g_k \Big( f_k + \sum_{i = 0}^{k-1} f_i g_i
    \prod_{j = i+1}^{k-1} (1 + g_j)  \Big)\\
    &= f_n + \sum_{k = 0}^{n-1} g_k f_k + \sum_{k = 0}^{n-1} \sum_{i =
    0}^{k - 1} f_i g_i g_k \prod_{j = i+1}^{k-1} (1 + g_j) \\
    &= f_n + \sum_{k = 0}^{n-1} g_k f_k + \sum_{i = 0}^{n-1} f_i g_i
    \sum_{k = i+1}^{n-1} g_k \prod_{j = i+1}^{k-1} ( 1 + g_j)\\
    &= f_n + \sum_{k = 0}^{n-1} f_k g_k \Big( 1 + \sum_{i = k + 1}^{n-1} g_i
    \prod_{j = k + 1}^{i - 1} ( 1 + g_j) \Big).    
  \end{align*}
  Thus, it suffices to show that
  \begin{align}
    \label{eq:eq1}
    1 + \sum_{i = k }^{n-1} g_i \prod_{j = k }^{i - 1} (1 + g_j ) 
    = \prod_{j = k}^{n-1} ( 1 +  g_j), \quad \text{ for all } 0
    \le k \le n - 1.
  \end{align}
  But this follows from a telescopic sum argument 
  as follows:
  \begin{align*}
    1 - \prod_{j = k }^{n-1} (1 + g_j) 
    &= \sum_{i = k }^{n-1} \Big( \prod_{j = k}^{i - 1} (1 + g_j)
    - \prod_{j = k}^{i} (1 + g_j) \Big) \\
    &= \sum_{i = k}^{n-1} \Big( \big( 1 - (1 + g_i) \big) \prod_{j = k}^{i-1}
    (1 + g_j) \Big)\\ 
    &= - \sum_{i = k }^{n-1} g_i \prod_{j = k}^{i-1} (1 + g_j). 
  \end{align*}
  Rearranging the terms yields \eqref{eq:eq1} and completes the proof.
\end{proof}

Next, we introduce the discrete time counter-part of
Proposition~1 in \cite{scheutzow2013}.


\begin{lemma}
  \label{lem:martineq}
  Let $(M_n)_{n \in \N_0}$ be an $(\F_n)_{n \in \N_0}$-martingale with $M_0 =
  0$. Then, for every $p \in (0,1)$ and every $n \in \N_0$ we have
  \begin{align}
    \label{eq:martineq1}
    \E [ (\sup_{0 \le k \le n} M_k)^p ] \le \frac{1}{1 - p} \big( \E [ - \inf_{0
    \le k \le n} M_k  \ ] \big)^p
  \end{align}
  or, equivalently, 
  \begin{align}
    \label{eq:martineq2}
    \E [ (\sup_{k \in \N_0} M_k)^p ] \le \frac{1}{1 - p} \big( \E [ - \inf_{ k
    \in \N_0} M_k  \ ] \big)^p.
  \end{align}
\end{lemma}

\begin{proof}
  The equivalence of \eqref{eq:martineq1} and \eqref{eq:martineq2} follows
  at once from the monotone convergence theorem and from stopping the
  martingale at $n$, respectively. Hence, it suffices to prove
  \eqref{eq:martineq1}.

  Since $M_0 = 0$ we get that
  \begin{align*}
    0 = \E[ M_n ] = \E [ (M_n \vee 0)]  - \E[ (- M_n) \vee 0 ], \quad
    \text{ for all } n \in \N_0,
  \end{align*}
  and, consequently,  
  \begin{align*}
    \E[ M_n \vee 0 ] = \E[ (- M_n) \vee 0 ] \le \E\big[ \sup_{0 \le k \le n}
    (-M_k) \big] = \E \big[ - \inf_{0 \le k \le n} M_k \big] 
  \end{align*}
  for all $n \in \N_0$.
  Next, for $n \in \N_0$ we define a mapping $\varphi_n \colon [0,\infty) \to
  [0,1]$ by
  \begin{align*}
    \varphi_n(x) = \P \big( \sup_{0 \le k \le n} M_k \ge x \big), \quad \text{
    for all } x \ge 0.
  \end{align*}
  Now, fix $x > 0$ and $n \in \N_0$ arbitrarily and define the stopping times
  \begin{align*}
    \tau_n := \inf \big\{ m \in \N_0 \, : \, M_m \ge x \} \wedge n. 
  \end{align*}
  We set 
  \begin{align*}
    \widetilde{M}_k := M_{k \wedge \tau_n}
  \end{align*}
  and note that $(\widetilde{M}_k)_{k \in \N_0}$ is again an $(\F_k)_{k \in
  \N_0}$-martingale with $\widetilde{M}_0 = M_0 = 0$. In addition, we have $\E
  [\widetilde{M}_k \vee 0] = \E[ (- \widetilde{M}_k) 
  \vee 0 ] \le \E [ - \inf_{0 \le \ell \le (n \wedge k)} M_\ell ]$ for all $k
  \in \N_0$ and the same $n \in \N_0$ as above.
  Furthermore, note that $\{ \sup_{0 \le k \le n} M_k \ge x \} =
  \{ \sup_{0 \le k \le n} \widetilde{M}_k \ge x \} = \{ \widetilde{M}_n \ge x
  \}$ and, therefore,
  \begin{align*}
    \E [ \widetilde{M}_n \vee 0] \ge x \varphi_n(x).
  \end{align*}
  Altogether, this implies
  \begin{align*}
    \varphi_n(x) \le \frac{1}{x} \E [ \widetilde{M}_n \vee 0] \le
    \frac{1}{x} \E \big[ - \inf_{0 \le \ell \le n} M_\ell \big].
  \end{align*}
  Finally, we obtain for every $p \in (0,1)$
  \begin{align*}
    \E \big[ (\sup_{0 \le k \le n} M_k )^{p} \big] &= \int_0^\infty \P \big(
    \sup_{0 \le k \le n} M_k \ge x^{\frac{1}{p}} \big) \diff{x}\\
    &\le \int_{0}^{\infty} \big( x^{-\frac{1}{p}} \E \big[ - \inf_{0 \le \ell
    \le n} M_\ell \big] \big) \wedge 1 \diff{x}\\
    &= \big( \E \big[ - \inf_{0 \le \ell \le n} M_\ell \big] \big)^p +
    \frac{p}{1 - p} \big( \E \big[ - \inf_{0 \le \ell \le n} M_\ell \big]
    \big)^p\\ 
    &= \frac{1}{1 - p} \big( \E \big[ - \inf_{0 \le \ell \le n} M_\ell \big]
    \big)^p, 
  \end{align*}
  which is the assertion.
\end{proof}

\begin{remark}
  The constant $\frac{1}{1-p}$ in Lemma~\ref{lem:martineq} is (most likely) not
  sharp but at most off from the optimal constant $C_p$ by a factor of
  $\frac{4}{\pi}$. This can be seen as follows: Let $W \colon [0, \infty)
  \times \Omega \to \R$ be a 
  standard Wiener process and define the stopping time $\tau_{-1} = \inf \{ s
  \ge 0 \, : \, W(s) = -1 \}$. Set $\widehat{M}(t) := W(t \wedge
  \tau_{-1})$ and $h_k = 2^{-k}$ for $k \in \N_0$. Then, for every $k \in \N$
  we obtain a discrete time martingale by setting $M_n^k := \widehat{M}(nh_k)$.
  From the continuity of the trajectories of the Wiener process
  and the monotone convergence theorem it follows that 
  \begin{align*}
    \lim_{k \to \infty} &\E \big[ \big(\sup_{ \ell \in \N_0 } M_\ell^k \big)^p
    \big]   
    = \E \big[ \big( \sup_{t \ge 0} \widehat{M}(t) \big)^p
    \big]=\int_0^{\infty} \P \big( 
    \sup_{t \ge 0} \widehat{M}(t) \ge x^{\frac{1}{p}} \big) \diff{x}\\
    &=\int_0^{\infty} \big(1+x^{1/p}\big)^{-1} \diff{x}=\frac{\pi p}{\sin (\pi
    p)} =\frac{\pi p}{\sin (\pi p)} \lim_{k \to \infty}  
    \E \big[ \big(-\inf_{ \ell \in \N_0 } M_\ell^k \big)^p
    \big].
  \end{align*}
  Hence,
  $$
  \frac{\pi p}{\sin (\pi p)} \le C_p \le \frac 1{1-p}.
  $$
  The ratio $R_p:=\frac 1{1-p} \frac{\sin (\pi p)}{\pi p}$ is easily seen to
  obtain its maximum value $\frac{4}{\pi}$ at $p=1/2$. Note that 
  $$
  \lim_{p \downarrow 0} R_p= \lim_{p \uparrow 1} R_p=1,
  $$
  so the constant $\frac{1}{1-p}$ in Lemma \ref{lem:martineq} becomes optimal
  in the limits $p \to 0$ and $p \to 1$. 
\end{remark}

Now we are well-prepared for the proof of Theorem~\ref{th:gronwall}:

\begin{proof}[Proof of Theorem~\ref{th:gronwall}]
  We first apply Lemma~\ref{lem:detGronw} $\omega$-wise and
  obtain 
  \begin{align*}
    X_n &\le F_n + M_n + \sum_{k = 0}^{n-1}
    \big(F_k + M_k \big) G_k \prod_{j = k + 1}^{n-1} (1 + G_j) \\
    &= F_n + \sum_{k = 0}^{n-1} F_k G_k \prod_{j = k + 1}^{n-1} (1 + G_j)
    + M_n + \sum_{k = 0}^{n-1} M_k G_k \prod_{j = k + 1}^{n-1} (1 + G_j).
  \end{align*}
  Now, since $F_k \le F_n^{\ast}$ for all $k \le n$ we have
  \begin{align*}
    F_n + \sum_{k = 0}^{n-1} F_k G_k \prod_{j = k + 1}^{n-1} (1 + G_j)
    &\le F_n^\ast \Big( 1 + \sum_{k = 0}^{n-1} G_k \prod_{j = k + 1}^{n-1} ( 1+
    G_j) \Big)\\
    &= F_n^\ast \prod_{j = 0}^{n-1} (1 + G_j),
  \end{align*}
  where we applied \eqref{eq:eq1}. Moreover, it holds true that
  \begin{align*}
    \sum_{k = 0}^{n-1} M_k G_k \prod_{j = k + 1}^{n-1} ( 1 + G_j) &=
    \sum_{k = 0}^{n-1} M_k (1 + G_k - 1) \prod_{j = k + 1}^{n-1} ( 1 + G_j)\\
    &= \sum_{k = 0}^{n-1} M_k \Big( \prod_{j = k }^{n-1} ( 1 + G_j) - \prod_{j =
    k + 1}^{n-1} ( 1 + G_j) \Big). 
  \end{align*}
  Hence, since $M_0 = 0$ we get by summation by parts
  \begin{align*}
    M_n + \sum_{k = 0}^{n-1} M_k G_k \prod_{j = k + 1}^{n-1} (1 + G_j)
    &= \sum_{k = 0}^{n-1} \big( M_{k + 1} - M_k \big) \prod_{j =
    k+1}^{n-1} (1 + G_j)\\
    &= L_n \prod_{i = 0}^{n-1} (1 + G_i),
  \end{align*}
  where 
  \begin{align*}
    L_n := \sum_{k = 0}^{n-1} ( M_{k + 1} - M_k ) \prod_{j = 0}^k (1 +
    G_j)^{-1} 
  \end{align*}
  is a further $(\F_n)_{n \in \N_0}$-martingale. Altogether, we have shown that 
  \begin{align}
    \label{eq:ineq_X}
    X_n \le \big( F_n^\ast + L_n \big) \prod_{i = 0}^{n-1} (1 + G_i).
  \end{align}
  Hence, H\"older's inequality with $1 = \frac{1}{\mu} +
  \frac{1}{\nu}$ yields
  \begin{align*}
    \E \big[ \sup_{0 \le k \le n} X_k^p \big] &\le \Big\| \prod_{i = 0}^{n-1}
    (1 + G_i)^p \Big\|_{L^{\mu}(\Omega)} \big\| ( F_n^\ast + L_n^\ast )^p
    \big\|_{L^\nu(\Omega)}\\
    &\le \Big\| \prod_{i = 0}^{n-1} (1 + G_i)^p \Big\|_{L^{\mu}(\Omega)}
    \Big( \E \big[ (F_n^\ast)^{\nu p} \big] + \E \big[ (L_n^\ast )^{\nu p}
    \big] \Big)^{\frac{1}{\nu}}.   
  \end{align*}
  Moreover, since $X_n \ge 0$ it follows from \eqref{eq:ineq_X} 
  that $- L_n \le F_n^\ast$ for all $n \in \N_0$. Therefore, we have $- \inf_{0
  \le k \le n} L_k \le F_n^\ast$. Thus, after applying the martingale
  inequality from Lemma~\ref{lem:martineq} to $\E \big[ (L_n^\ast )^{\nu p}
  \big]$ we conclude 
  \begin{align*}
    \E [ \sup_{0 \le k \le n} X_k^p ] \le 
    \Big\| \prod_{i = 0}^{n-1} (1 + G_i)^p \Big\|_{L^{\mu}(\Omega)}
    \Big( \E \big[ ( F_n^\ast)^{\nu p} \big]  + \frac{1}{1 - \nu p} \big( \E
    \big[ F_n^\ast \big] \big)^{\nu p} \Big)^{\frac{1}{\nu}}.
  \end{align*}
  An application of Jensen's inequality completes the proof of
  Equation~\eqref{eq:gron1}. The proof of Equation~\eqref{eq:gron2} follows
  from the same steps but with $\mu = \infty$.
\end{proof}

\section{Application to numerical schemes}
\label{sec:appl}

In this section we prove an a priori estimate for the backward Euler-Maruyama
approximation of solutions to stochastic differential equations, whose
coefficient functions satisfy a coercivity condition.

To be more precise let $T > 0$ and $d,m \in \N$. Consider the stochastic
ordinary differential equation 
\begin{align}
  \label{eq:SDE}
  \begin{split}
    \diff{X(t)} &= f(X(t)) \diff{t} + g(X(t)) \diff{W(t)}, \quad t \in [0,T],\\
    X(0) &= X_0,
  \end{split}
\end{align}
where $f \colon \R^d \to \R^d$ and $g \colon \R^d \to \R^{d \times m}$ denote
the drift and diffusion coefficient functions, respectively. Further, $W \colon
[0,T] \times \Omega \to \R^m$ is a standard Wiener process on a filtered
probability space $(\Omega, \F, (\F_t)_{t \in [0,T]}, \P)$. For simplicity, let
the initial condition $X_0 \in \R^d$ be deterministic.

We assume that $f$ and $g$ are continuous and satisfy the following coercivity
condition: There exists $L \ge 0$ such that
\begin{align}
  \label{eq:coerc}
  \langle f(x) , x \rangle + \frac{1}{2} | g(x) |^2 \le L \big( 1 + |x|^2 \big)
\end{align}
for all $x \in \R^d$, where we let $|\cdot|$ denote the Euclidean norms on
$\R^d$ and $\R^m$ as well as the Frobenius norm if applied to matrices from
$\R^{d \times m}$.  

An often considered numerical method for the approximation of the solution $X$
to \eqref{eq:SDE} is the \emph{backward Euler-Maruyama method}, see for
instance \cite{kloeden1992, milstein1995}, given by
\begin{align}
  \label{eq:BEM}
  \begin{split}
    Y^{j+1} &= Y^j + h f(Y^{j+1}) + g(Y^j) \Delta_h W^{j+1}, \quad j =
    1,\ldots,N_h,\\
    Y^0 &= X_0,
  \end{split}
\end{align}
where $h \in (0,1)$ denotes the equidistant \emph{step size} and $N_h \in \N$
is determined by $N_h h \le T < (N_h + 1) h$. The stochastic increment is given
by $\Delta_h W^{j+1} = W( t_{j+1} ) - W(t_j)$, where $t_j = jh$.

Our aim is to prove the following \emph{a priori estimate} on
$(Y^{j})_{j = 0}^{N_h}$, which is a sharper version of Theorem~4.2 in
\cite{andersson2015}  in the sense that taking the supremum now occurs inside
the expectation but only with respect to the $L^{2p}$-norm for $p \in (0,1)$.

\begin{prop}
  Let $h_0 \in (0, (2L)^{-1} )$ denote an upper step size bound. For every $p
  \in (0, 1)$ and for every
  $(\F_{nh})_{n \in \N_0}$-adapted process
  $(Y^n)_{n \in \N_0}$ satisfying \eqref{eq:BEM} with $h \in 
  (0,h_0)$ we have
  \begin{align*}
    \E \big[ \sup_{0 \le j \le N_h} \big( | Y^j
    |^{2} +  h |g(Y^j)|^2 \big)^p \big] &\le 
    \Big(1 + \frac{1}{1 - p} \Big)
    \exp\big( p (1 - 2 h_0 L)^{-1} 2 L T  \big)\\
    &\quad \times \Big( |X_0|^2 + (1 - 2 h_0 L)^{-1}
    \big( h_0 |g(X_0)|^2 + 2 L T \big) \Big)^p.
  \end{align*}
  In particular, this bound is independent of the step size $h$.
\end{prop}

\begin{proof}
  Let $(Y^n)_{n \in \N_0}$ be an adapted process satisfying \eqref{eq:BEM}
  with step size $h \in (0,h_0)$. For every $j \in \{0, \ldots, N_h
  - 1\}$ we get from the polarization identity $\langle a - b, a \rangle =
  \frac{1}{2} ( |a|^2 - |b|^2 + |a - b |^2 )$, which is valid for all $a, b \in
  \R^d$, that
  \begin{align*}
      | Y^{j+1} |^2 - | Y^j |^2 + | Y^{j+1} - Y^j |^2 &= 2 \langle Y^{j+1} -
      Y^j, Y^{j+1} \rangle \\
      &= 2 h \langle f (Y^{j+1}), Y^{j+1} \rangle + 2 \langle g(Y^j) \Delta_h 
      W^{j+1}, Y^{j+1} \rangle, 
  \end{align*}
  since $Y^{j+1}$ satisfies \eqref{eq:BEM}.
  Now, an application of the coercivity condition \eqref{eq:coerc} yields
  \begin{align}
    \label{eq:ineq5}
    \begin{split}
      &| Y^{j+1} |^2 - | Y^j |^2 + | Y^{j+1} - Y^j |^2\\
      &\quad \le 2 h L \big( 1 + |Y^{j+1}|^2 \big) - h | g(Y^{j+1}) |^2 
      + 2 \langle g(Y^j) \Delta_h W^{j+1}, Y^{j+1} - Y^j \rangle \\
      &\qquad + 2 \langle g(Y^j) \Delta_h W^{j+1}, Y^j \rangle.
    \end{split}
  \end{align}
  From the Cauchy-Schwarz and Young inequalities we deduce
  \begin{align*}
    2 \langle g(Y^j) \Delta_h W^{j+1}, Y^{j+1} - Y^j \rangle \le | g(Y^j)
    \Delta_h W^{j+1}|^2 + |Y^{j+1} - Y^j|^2.
  \end{align*}
  Note that the second term also appears on the left hand side of the
  inequality \eqref{eq:ineq5}. After cancelling and some rearranging we
  therefore get 
  \begin{align}
    \label{eq:ineq_j}
    |Y^{j+1}|^2 + h | g(Y^{j+1})|^2  \le |Y^j|^2 + h |g(Y^j) |^2 + 2
    h L \big( 1  + |Y^{j+1}|^2 \big) + Z^{j+1},
  \end{align}
  where
  \begin{align}
    \label{eq:Z}
    Z^{j+1} := | g(Y^j) \Delta_h W^{j+1}|^2 - h | g(Y^j) |^2 + 2 \langle g(Y^j)
    \Delta_h W^{j+1}, Y^j \rangle.
  \end{align}
  By iterating the inequality we arrive at
  \begin{align*}
    |Y^n|^2 + h |g(Y^n)|^2 &\le |Y^0|^2 + h |g(Y^0)|^2 + 2 h L \sum_{j =
    0}^{n-1} \big(1 + |Y^{j+1}|^2 \big) + \sum_{j = 0}^{n-1} Z^{j+1},
  \end{align*}
  or, equivalently,
  \begin{align}
    \begin{split}
      \label{eq:ineq_n}
      &(1 - 2 h L)|Y^n|^2 + h |g(Y^n)|^2\\
      &\quad \le (1 - 2 h L) |Y^0|^2 + h |g(Y^0)|^2 + 
      2 L t_n  + \sum_{j = 0}^{n-1} Z^{j+1} 
      + 2 h L \sum_{j = 0}^{n-1} |Y^j|^2.
    \end{split}
  \end{align}
  Next, note that $1 \ge (1 - 2 h L) \ge (1 - 2 h_0 L) > 0$. From this we
  finally obtain the relationship
  \begin{align*}
    X_n \le F_n + M_n + \sum_{j = 0}^{n-1} G_j X_j, \quad \text{ for all } n
    \in \N_0, 
  \end{align*}
  where 
  \begin{align*}
    X_n &:= |Y^n|^2 + h |g(Y^n)|^2, \\
    F_n &:= |Y^0|^2 + (1 - 2 h_0 L)^{-1} \big( h_0 |g(Y^0)|^2 + 2 L t_n \big),\\
    M_n &:= (1 - 2 h_0 L)^{-1} \sum_{j = 0}^{n-1} Z^{j+1},\\
    G_n &:= (1 - 2 h_0 L)^{-1} 2 h L,
  \end{align*}
  for all $n \in \N_0$. Clearly, the processes $(X_n)_{n \in \N_0}$,
  $(F_n)_{n \in \N_0}$, and $(G_n)_{n \in \N_0}$ satisfy the assumptions of
  Theorem~\ref{th:gronwall}. Hence it remains to show that $(M_n)_{n \in \N_0}$
  is a martingale with respect to the filtration $(\F_{t_n})_{n \in \N_0}$.

  For this first note that $(M_n)_{n \in \N_0}$ is adapted and satisfies $M_0 =
  0$. Then, we show inductively that $M_n =(1 - 2 h_0 L)^{-1} \sum_{j =
  0}^{n-1} Z^{j+1}$ as well as the random variables $|Y^n|^2$, $|g(Y^n)|^2$ are
  integrable: For $n = 0$ this is evident. Assume now that $|Y^{j}|^2$,
  $|g(Y^{j})|^2$ are 
  integrable for all $0 \le j < n$. Then, from \eqref{eq:Z}, the
  Cauchy-Schwarz and the Young inequality, we obtain the estimate
  \begin{align*}
    \E\big[ |Z^{n}| \big] \le 2 \E\big[ | g(Y^{n-1}) \Delta_h W^{n}|^2 \big]
    + h \E \big[ |g(Y^{n-1})|^2 \big] + \E \big[ |Y^{n-1}|^2 \big].
  \end{align*}
  The first term is bounded by the It\=o isometry by
  \begin{align*}
    2 \E\big[ | g(Y^{n-1}) \Delta_h W^{n}|^2 \big] = 2 h \E\big[ | g(Y^{n-1})
    |^2 
    \big],
  \end{align*}
  since $Y^{n-1}$ is independent of $\Delta_h W^{n}$. The latter two terms are
  bounded by the induction hypothesis. Altogether, this shows that
  $Z^{n}$ and, hence, $M_{n}$ are integrable random variables. Further, we have 
  \begin{align*}
    \E \big[ Z^{n} \big] = 0.
  \end{align*}
  Thus, taking expectation in \eqref{eq:ineq_n} yields that $|Y^n|^2$,
  $|g(Y^n)|^2$ are also integrable.

  Finally, as above we get
  \begin{align*}
    \E \big[ Z^{n} | \F_{t_{n-1}} \big] = 0,
  \end{align*}
  which proves the martingale property for $(M_n)_{n \in \N_0}$. 
  Theorem~\ref{th:gronwall} is therefore applicable and yields the assertion
  (together with the inequality $1 + x \le \ee^x$).
\end{proof}

\section*{Acknowledgement}

The first author gratefully acknowledges financial support by the research
center \textsc{Matheon}.


\begin{thebibliography}{1}

\bibitem{andersson2015}
A.~Andersson and R.~Kruse.
\newblock Mean-square convergence of the {BDF2}-{M}aruyama and backward {E}uler
  schemes for {SDE} satisfying a global monotonicity condition.
\newblock {\em Preprint, arXiv:1509.00609}, 2015.

\bibitem{banuelos2013}
R.~Ba{\~n}uelos and A.~Os{\c{e}}kowski.
\newblock Sharp maximal {$L^p$}-estimates for martingales.
\newblock {\em Illinois J. Math.}, 58(1):149--165, 2014.

\bibitem{bellman1943}
R.~Bellman.
\newblock The stability of solutions of linear differential equations.
\newblock {\em Duke Math. J.}, 10:643--647, 1943.

\bibitem{emmrich1999}
E.~Emmrich.
\newblock Discrete versions of {G}ronwall's lemma and their application to the
  numerical analysis of parabolic problems.
\newblock {\em TU Berlin, Fachbereich Mathematik, Preprint}, 637-1999, 1999.
\newblock
\href{http://www3.math.tu-berlin.de/preprints/files/Preprint-637-1999.pdf}{(Link
to PDF)}

\bibitem{gronwall1919}
T.~H. Gronwall.
\newblock Note on the derivatives with respect to a parameter of the solutions
  of a system of differential equations.
\newblock {\em Ann. of Math. (2)}, 20(4):292--296, 1919.

\bibitem{kloeden1992}
P.~E. Kloeden and E.~Platen.
\newblock {\em Numerical Solution of Stochastic Differential Equations},
  volume~23 of {\em Applications of Mathematics (New York)}.
\newblock Springer-Verlag, Berlin, 1992.

\bibitem{milstein1995}
G.~N. Milstein.
\newblock {\em Numerical integration of stochastic differential equations},
  volume 313 of {\em Mathematics and its Applications}.
\newblock Kluwer Academic Publishers Group, Dordrecht, 1995.
\newblock Translated and revised from the 1988 Russian original.

\bibitem{scheutzow2013}
M.~Scheutzow.
\newblock A stochastic {G}ronwall lemma.
\newblock {\em Infin. Dimens. Anal. Quantum Probab. Relat. Top.}, 16(2):4,
  2013.

\end{thebibliography}

\end{document}